\documentclass[a4paper, 11pt]{article}

\usepackage{ulem}

\usepackage{amsmath, amssymb}
\usepackage{mathrsfs}
\usepackage{thm-restate}

\usepackage{amsthm}
\usepackage{color}
\usepackage{graphics}
\usepackage{mathtools}

\newtheorem{defn}{Definition}[section]  
\newtheorem{rem}[defn]{Remark}

\theoremstyle{plain}

\newtheorem{prop}[defn]{Proposition} 
\newtheorem{lem}[defn]{Lemma} 
\newtheorem{cor}[defn]{Collorary}

\newtheorem*{thm*}{Theorem}



\newenvironment{pf*}[1]{\proof[#1]}{\endproof}

\newcommand{\myclaim}{問題}
\newtheorem{claim}[defn]{\myclaim}

\newcommand{\lm}{\lambda}
\newcommand{\llm}{{\lambda}'}
\newcommand{\Lm}{\Lambda}
\newcommand{\fp}{\mathbb{F}_p}
\newcommand{\fpp}{\mathbb{F}_{p^2}}
\newcommand{\fq}{\mathbb{F}_q}

\newcommand{\bQ}{\mathbb{Q}}
\newcommand{\bN}{\mathbb{N}}
\newcommand{\bZ}{\mathbb{Z}}

\newcommand{\super}{\text{supersingular}}

\newcommand{\height}{\operatorname{ht}}

\title{The Lang-Trotter conjecture on average for genus-$2$ curves with Klein-$4$ reduced automorphism group }
\author{Chihiro Ando\thanks{Graduate School of Environment and Information Sciences, Yokohama National University.
E-mail: \texttt{ando-chihiro-zs@ynu.jp}}}

\begin{document}

\maketitle

\begin{abstract}
    \noindent 
    For an elliptic curve $E$ over $\mathbb{Q}$ without complex multiplication, the Lang--Trotter conjecture is that
    \[
    \# \{ p<X \mid E \text{ has a supersingular reduction at } p \} \sim \frac{c\sqrt{X}}{\log X}
    \]
    as $X \rightarrow \infty$, where $c>0$ is a constant depending only on $E$. \textcolor{black}{Fouvry} and Murty \textcolor{black}{obtained} an average estimation related to the Lang--Trotter conjecture, called the Lang--Trotter conjecture on average. We consider the Lang--Trotter conjecture for curves of genus 2 and obtain a similar result to the Lang--Trotter conjecture on average for the family of curves $C_{\lambda}:y^2=x(x-1)(x+1)(x-{\lambda})(x-1/ \lambda)$. These curves are characterized as curves of genus two with a reduced automorphism group containing the Klein $4$-group. 
    
\end{abstract}

\section{Introduction}

Let $E_{a,b}$ be the elliptic curve defined by the Weierstrass equation 
\[
y^2=x^3+ax+b 
\]
for $a,b \in \bQ$. Let ${\pi}_0(X,a,b)$ be the number of primes $p<X$ such that $E_{a,b}$ has a supersingular reduction at $p$. For $E_{a,b}$ with complex multiplication, Deuring \cite{Deu} showed that 
\[
{\pi}_0(X,a,b) \sim \frac{1}{2} \frac{X}{\log X}.
\]
For $E$ without complex multiplication, Lang and Trotter \cite{LT} \textcolor{black}{conjectured} the existence of a constant $c$ depending only on $E_{a,b}$ such that
\[
{\pi}_0(X,a,b) \sim \frac{c \sqrt{X}}{\log X}
\]
as $X \rightarrow \infty$. Although it is still an open question, a weaker result due to Elkies \cite{Elkies} is that there are infinitely many supersingular primes. 
In addition, \textcolor{black}{Fouvry} and Murty obtained an average estimation related to the Lang--Trotter conjecture. They showed that for $A,B > X^{1+ \epsilon}$, 
\[
\frac{1}{AB} \sum_{\lvert a \rvert \le A} \sum_{\lvert b \rvert \le B} {\pi}_{0}(X,a,b) \sim \frac{4 \pi}{3} \frac{\sqrt{X}}{\log X},
\]
where $a,b$ run  through the integers. They used the fact that the number of supersingular elliptic curves $E_{a,b}$ over $\fp$ ($p \neq2,3$) with $0 \le a,b < p $ is equal to   
\[
\frac{p}{2} \cdot \left(h(-4p)+h(-p) \right)+O(p),
\]
where $h(-p)$ and $h(-4p)$ are the class numbers of primitive quadratic forms with discriminants $-p$ and $-4p$, respectively.
Moreover, they proved that the average estimation is true if $A,B > X^{1/2+ \epsilon}$ and $AB > X^{3/2 +\epsilon}$.
This result is called \textit{the Lang--Trotter conjecture on average}. 
Let \textcolor{black}{\[ a_p(E):=p+1-\#E(\fp).\]} Then, \textcolor{black}{Fouvry} and Murty's result is the case in which $a_p(E)=0$. David and Pappalardi \cite{DP} generalized this result to the case in which $a_p(E)=r$ for any $r \in \bZ$, and Baier \cite{SB} improved the error terms in their results.  
 
In this paper, we extend the Lang--Trotter conjecture on average to curves of genus 2. \textcolor{black}{Let $K$ be an algebraically closed field of characteristic $p \ne 2$.} Curves of genus 2 \textcolor{black}{over $K$}  are classified into 7 types (cf. Igusa \cite{Igusa}; see Section \ref{subsec:2.2} for a review). In particular, we focus on one of them: the genus-2 curves defined by the following equation \textcolor{black}{with parameter $\lm \in K$}:
\[
C_{\lm}:y^2=x(x-1)(x+1)(x-{\lambda}) \left(x-\frac{1}{\lambda}\right).
\]
 These genus-two curves are characterized as having the reduced automorphism group ${\rm RA}(C_\lambda)$ containing the Klein 4-group $V_4$. \textcolor{black}{For almost all $\lambda$} , it is known that ${\rm RA}(C_\lambda)$ is isomorphic to $V_4$ \textcolor{black}{(cf. Igusa \cite[\S 8]{Igusa}, Ibukiyama, Katsura and Oort \cite[\S 1]{IKO02})}. 
Here, the \textit{reduced automorphism group} of \textcolor{black}{a} hyperelliptic curve $C$ is the quotient of the automorphism group of $C$ by the central subgroup generated by hyperelliptic involution.
Let $p$ be a prime such that $p \ge 5$. We consider whether $C_{\lm}$ has a superspecial reduction at $p$.
Here, a curve $C$ over a field $K$ is \textit{superspecial} if its Jacobian variety is isomorphic to a product of supersingular elliptic curves over $\overline{K}$, the algebraic closure. 

For superspecial genus-2 curves, Ibukiyama, Katsura and Oort \cite{IKO02} determined the number of isomorphism classes of superspecial genus-2 curves for each of the 7 types. Ibukiyama and Katsura \cite{IKadd} also described the number of isomorphism classes of superspecial genus-2 curves that have models over $\fp$ in terms of a class number and a type number. In addition, Katsura and Oort \cite{KatsuraOort} counted the number of $(a,b) \in {\overline{\fp}}^2$ 
such that the genus-2 curve $C_{a,b}:y^2=(x^2-1)(x^2-a)(x^2-b)$ is superspecial
and has a reduced automorphism group $V_4$.
However, they did not count the number of $(a,b) \in {\fp}^2$ 
such that the genus-2 curve $C_{a,b}$ is superspecial, nor the number of $\lambda\in\fp$ such that $C_\lambda$ is superspecial.
In this paper, we first determine the latter number to consider the Lang--Trotter conjecture on average for curves $C_{\lm}$. We set
\[
{\psi}_{p} := \# \lbrace \lm \in \fp\smallsetminus \{ 0, \pm1\} \mid C_{\lm} \text{ is superspecial}  \rbrace
\]
for a fixed $p$ and determine it. Next, for a fixed $\lm \in \bZ$ and a curve $C_{\lm}$, we set
\[
{\phi}_{\lm}(X):= \# \lbrace p<X \mid C_{\lm} \text{ has a superspecial reduction at } p  \rbrace .
\]
Using the value of $\psi_p$, we compute $\sum_{\lvert \lm \rvert \le N} {\phi}_{\lm}(X)$ for $N \ge X^{1+\epsilon}$. The first theorem determines $\psi_p$ in terms of class numbers $h(-p)$ and $ h(-4p)$. 

\begin{restatable}{theor}{thmA} \label{thm:A}
Let $p$ be prime with $p \ge 5$. Then, we have
   \begin{enumerate}
        \item[(1)] ${\psi}_{p}=h(-4p)$ if $p \equiv 1 \bmod 4$,
        \item[(2)] ${\psi}_{p}=6h(-p)-2$ if $p \equiv 3 \bmod 8$,
        \item[(3)] ${\psi}_{p}=2h(-p)-2$ if $p \equiv 7 \bmod 8$.
    \end{enumerate}
\end{restatable}   

 The proof uses the result of Ibukiyama, Katsura and Oort \cite[\S 1.2]{IKO02} that $C_{\lm}$ is superspecial if and only if two elliptic curves 
 \begin{equation}\label{eq:E_Lambda}
     E_{\Lm}:y^2=x(x-1)\left(x+(\lm - \sqrt{{\lm}^2-1})^2 \right)
 \end{equation}
 and 
\begin{equation}\label{eq:E_LambdaPrime}
 E_{{\Lm}'}:y^2=x(x-1)\left(x+(\lm +\sqrt{{\lm}^2-1})^2 \right)
 \end{equation}
 are supersingular. We make a one-to-one correspondence between a pair of superspecial curves $\{C_{\lm},C_{-\lm}\}$ and a pair of supersingular elliptic curves $\{ E_{\Lm},E_{{\Lm}'}\}$ and determine the value of ${\psi}_{p}$; see Section \ref{theorema}. This enables us to reduce the problem to that for elliptic curves.
However,
Theorem \ref{thm:A} does not immediately follow from known results for elliptic curves. The difficulties arise from the fact that the descriptions in \eqref{eq:E_Lambda} and \eqref{eq:E_LambdaPrime} require the field extension $\mathbb F_{p^2}/\mathbb F_{p}$ and that the set\textcolor{black}{s} of their $j$-invariants \textcolor{black}{are} proper subset\textcolor{black}{s} of the set\textcolor{black}{s} of supersingular $j$-invariants since the elliptic curves are in a specific form.
Thus, we require elaborate arguments to obtain Theorem \ref{thm:A}; see Section \ref{sec:ssing_and_classnum}. Regarding the enumeration of superspecial curves in a specific form, this result is reminiscent of the result of Auer and Top \cite[Proposition 3.2]{AT02} on the enumeration of supersingular elliptic curves in the Legendre form. 

In the second theorem, we compute $\sum_{\lvert \lm \rvert \le N} {\phi}_{\lm}(X)$ via the result of Theorem A. 
\begin{restatable}{theor}{thmB} \label{thm:B}
    Let $X,\epsilon$ be positive numbers, and let $N$ be a positive number that satisfies $N > X^{1+\epsilon}$. Then, we have
    \[
   \frac{1}{N} \sum_{\lvert \lm \rvert \le N} {\phi}_{\lm}(X) \sim \frac{3\pi}{2} \frac{\sqrt{X}}{\log X} 
    \]
    as $X \rightarrow \infty$.
\end{restatable}

We use the idea of \textcolor{black}{Fouvry} and Murty. First, we change the order of \textcolor{black}{summations} in the following way.
\begin{eqnarray}
\label{changesum}
\sum_{|\lm| \le N} {\phi}_{\lm}(X) =  \sum_{|\lm| \le N} \sum_{\substack{p < X \\  C_{\lm} \text{ is superspecial} }} 1 
= \sum_{p < X} \sum_{\substack{|\lm| \le N \\ C_{\lm} \text{ is superspecial} }} 1 \textcolor{black}{.}
\end{eqnarray}
If $\lm \equiv \llm \bmod p$, then $C_{\lm}$ and $C_{\llm}$ are the same curve in $\fp$, and in particular, if one is superspecial, the other is also superspecial. Hence, for every $p \ge 5$, we describe
the inner sum of the right-hand side of \eqref{changesum} in terms of ${\psi}_p$.
From the result of Theorem \ref{thm:A}, we can express $\sum_{|\lm| \le N} {\phi}_{\lm}(X)$ \textcolor{black}{in terms of} class numbers, and we can calculate it by using the class number formula.

In addition to Theorem \ref{thm:B}, we have a similar result for $\lm$ running through rational numbers whose height is less than a 
 positive number $N > X^{1+ \epsilon}$. Here, for $\lm = b/a$ with coprime integers $a,b$, the \textit{height} of $\lm$ is defined as
 $\text{max}\{|a|,|b|\}$, denoted by $\height(\lm)$.   

 \begin{restatable}{theor}{thmC} \label{thm:C}
      Let $X,\epsilon$ be positive numbers and $N$ be a positive number that satisfies $N > X^{1+\epsilon}$. Then,
    \[
   \frac{1}{N^2} \sum_{\lm \in \bQ,\ \height(\lm) \le N } {\phi}_{\lm}(X) \sim \frac{9}{\pi} \frac{\sqrt{X}}{\log X} 
    \]
    as $X \rightarrow \infty$.
 \end{restatable}
 
For each $t \in \fp$, we calculate the number of $\lm \in \bQ $ such that $\height(\lm) \le N$ and \textcolor{black}{$\lm \equiv t \bmod p$}. Then, we use the result of Theorem B to deduce Theorem C.

The following is an outline of this paper. In Section 2, we recall some properties of supersingular elliptic curves and genus-two curves, which we use in the proof of Theorem \ref{thm:A}. In addition, we recall an analytic result that we use in the proof of Theorem \ref{thm:B}. In Section 3, we express the number of supersingular elliptic curves in a specific form with the class numbers $h(-p),h(-4p)$. Sections 4, 5 and 6 give the proofs of Theorems \ref{thm:A}, \ref{thm:B}, and \ref{thm:C}, respectively.

\section{Preliminaries}
In this section, we recall some known results on supersingular elliptic curves and  their relations with the class numbers. We also review some properties of genus-two curves. Moreover, we recall an analytic result for the Legendre symbol.
  \subsection{Supersingular Legendre elliptic curves}\label{subsec:2.1}
  We review some results on elliptic curves by Auer and Top in \cite{AT02} and we describe some relations between supersingular elliptic curves and the class numbers.

  Let $E_{\lm}$ 
  be a Legendre elliptic curve $y^2=x(x-1)(x-\lm)$ for $\lambda\in\fp$.
  We assume that $j(E_{\lm}) \ne 0$.
  We determine which Legendre elliptic curves are isomorphic to $E_{\lm}$ over $\fp$. 
  We say an elliptic curve $E$ over $\fp$ is \textit{Legendre isomorphic} if $E$ is isomorphic over $\fp$ to a Legendre elliptic curve $E_{\lm}$ with $\lm \in \fp$. 
 
\begin{prop}[Auer and Top {\cite[Lemma 3.1.]{AT02}}]
\label{legendrecase}
    Let $p$ be a prime satisfying $p>3$ and \textcolor{black}{$p \equiv 3 \bmod 4$}. Let $E$ be an elliptic curve over $\fp$ with j-invariant $j(E) \ne 0$. If $E$ is Legendre isomorphic, then there are exactly 3 values of $\lm \in \fp \smallsetminus \{0,1\}$ such that $E \cong E_{\lm}$ over $\fp$. More precisely, we have
    \begin{itemize}
        \item  $E_{-1} \cong E_{2} \cong E_{1/2}$;
        \item  if $\lm \neq -1,2,1/2$, then $E_{\lm} \ncong E_{1-\lm},E_{1/\lm} \ncong E_{1-1/\lm},E_{1/(\lm-1)} \ncong E_{\lm/(\lm-1)}$ over $\fp$, and 
        \begin{enumerate}
        
            \item [] $E_{\lm} \cong E_{1/\lm} \cong E_{1/(1-\lm)} \ $ if $\lm \in {\fp}^2$ and $ \lm-1 \in {\fp}^2$,

            \item[]  $E_{\lm} \cong E_{1/\lm} \cong E_{\lm/(\lm-1)} \ $
            if $\lm \in {\fp}^2$ and $\lm-1 \notin {\fp}^2$,

            \item [] $E_{\lm} \cong E_{1-1/\lm} \cong E_{1/(1-\lm)} \ $ if $\lm \notin {\fp}^2$ and $ \lm-1 \in {\fp}^2$,

            \item[]  $E_{\lm} \cong E_{1-1/\lm} \cong E_{\lm/(\lm-1)} \ $
            if $\lm \notin {\fp}^2$ and $ \lm-1 \notin {\fp}^2$.
            
        \end{enumerate}
    \end{itemize}
\end{prop}

The next lemma tells us when $\lm$ and $\lm-1$ belong to ${\fp}^2$.

\begin{prop}[Auer and Top {\cite[Lemma 2.1.]{AT02}}]
\label{abginf2?}
    Let $E$ be an elliptic curve $y^2=x(x-1)(x-\lm)$ with $\lm \in \fp \smallsetminus\{0,1\}$. Then,
    \[
    (\lm,0)\in [2]E(\fp) \text{ if and only if } \lm,\lm-1 \in \left( {\fp}^{*} \right) ^2,
    \]
   \textcolor{black}{ Here, the set $[2]E(\fp)$ is the image of $E(\fp)$ by the 2-multiplication map $[2]$.}
\end{prop}

The next proposition shows that for a supersingular elliptic curve $E_{\lm}$, $\lm$ is in a specific form.

\begin{prop}[Auer and Top {\cite[Proposition 3.2.]{AT02}}]
\label{propfpp8}
     Let $p \ge 3$ be a prime and $q$ be a power of $p$. Let $E_{\lm}$ be a supersingular elliptic curve with $\lm \in \fq \smallsetminus \{0,1 \}$. Then $-\lm \in \left( {{\mathbb{F}}_{p^2}}^{*}\right) ^8$.
\end{prop}

We review the enumeration of supersingular elliptic curves up to $\fp$-isomorphism by the class numbers. Let $p>3$ be a prime. It is known that for a supersingular elliptic curve $E$ over $\fp$, the endomorphism ring ${\textrm{End}}_{\fp}(E)$ over $\fp$ is isomorphic to $\bZ[\sqrt{-p}] $ or $\bZ \left[\frac{1+\sqrt{-p}}{2} \right]$. Moreover, it is known that if $p \equiv 1 \bmod 4$, the number of $\fp$-isomorphism classes of supersingular elliptic curves over $\fp$ that satisfy 
${\textrm{End}}_{\fp}(E) \cong \bZ[\sqrt{-p}] $ is equal to $h(-4p)$, and if $p \equiv 3 \bmod 4$, the number of $\fp$-isomorphism classes of supersingular elliptic curves over $\fp$ that satisfy ${\textrm{End}}_{\fp}(E) \cong \bZ \left[\frac{1+\sqrt{-p}}{2} \right]$ is equal to $ h(-p)$ (cf. Delfs and Galbraith \cite[Theorem 2.1, Theorem 2.7]{DG02}).
\\

By the fact described above and Proposition \ref{legendrecase}, we have the following result of Auer and Top.
Let $T_p:=\{ s \in \fp\smallsetminus\{0,1 \} \mid E_s\text{ is supersingular} \}$ .

\begin{prop}[Auer and Top {\cite[Proposition 3.2.]{AT02}}]
\label{numtp}
    Let $p>3$ be a prime satisfying \textcolor{black}{$ p \equiv 3 \bmod 4$}. Then $\# T_p=3h(-p)$.     
\end{prop}

We can also deduce the next lemma.

\begin{lem}\label{spsub}
The number of $\mu \in  \fpp \smallsetminus \fp$
such that $y^2=x(x-\mu)(x-\overline{\mu})$ is supersingular is equal to $(p-1)h(-4p)$.
\end{lem}

\begin{proof}
    \textcolor{black}{Following a similar approach to} Auer and Top {\cite[Proposition 3.2.]{AT02}}, we see that for a supersingular elliptic curve $E:y^2=x(x-\mu)(x-\overline{\mu})$ with $\mu \in \fpp \smallsetminus \fp$, the endomorphism ring of $E$ over $\fp$ satisfies ${\text{End}}_{\fp} (E) \cong \bZ[\sqrt{-p}]$, and conversely, for a supersingular elliptic curve $E$ over $\fp$ with an endomorphism ring over $\fp$ satisfying ${\text{End}}_{\fp}(E) \cong \bZ[\sqrt{-p}]$, $E$ can be given by an equation $y^2=x(x-\mu)(x-\overline{\mu})$ with $\mu \in \fpp \smallsetminus \fp$. Since two elliptic curves $E_{1}:y^2=x(x-\mu)(x-\overline{\mu})$ and $E_{2}:y^2=x(x-\nu)(x-\overline{\nu})$ are isomorphic over $\fp$ if and only if $\mu = u^2\nu$ or $\mu=u^2\overline{\nu}$ for some $u \in {\fp}^{*}$, we conclude that
    the number of $\mu \in  \fpp \smallsetminus \fp$
such that $y^2=x(x-\mu)(x-\overline{\mu})$ is supersingular is
    $2 \cdot \# \left({\fp}^{*}\right)^2 \cdot h(-4p)=(p-1)h(-4p)$.
\end{proof}

\subsection{Genus-2 curves}\label{subsec:2.2}
Let $K$ be a field of characteristic $\ne 2$ and $\overline K$ the algebraic closure of $K$. We review some facts on genus-two curves. Igusa \cite{Igusa} classified curves of genus two into 7 types as follows: 
\begin{enumerate}
   \item[(0)]$y^2=x(x-1)(x-{\lm}_1)(x-{\lm}_2)(x-{\lm}_3)$
        \item [(1)]$y^2=x(x-1)(x-\lm)(x-\mu)\left(x-\lm\frac{1-\mu}{1-\lm} \right)$
        \item[(2)]$y^2=x(x-1)(x-\lm)\left(x-\frac{\lm-1}{\lm}\right) \left(x-\frac{1}{1-\lm}\right)$
        \item[(3)] $y^2=x(x-1)(x+1)(x-\lm)\left(x-\frac{1}{\lm}\right)$
        \item[(4)]$(p\neq 3,5) \ y^2=x(x-1)(x+1)(x-2)\left(x-\frac{1}{2}\right) $
        \item[(5)]$y^2=x(x^2-1)(x^2+1) $
        \item[(6)]$(p \neq 5) \ y^2=x(x-1)(x-1-\zeta)(x-1-\zeta-{\zeta}^2)(x-1-\zeta-{\zeta}^2-{\zeta}^3)$, where $\zeta$ is a primitive fifth root of unity
    \end{enumerate}

We focus on a genus-2 curves of the type (3):
\[
C_{\lm}:y^2=x(x-1)(x+1)(x-{\lambda}) \left(x-\frac{1}{\lambda}\right)
\]
 and recall some properties shown by Ibukiyama, Katsura and Oort in \cite{IKO02}.

For each $\lambda \in K$, we choose a square root (in $\overline K$) of $\lambda^2-1$, fix it and write it as $\sqrt{\lambda^2-1}$.
For the curve $C_{\lm}$, we consider two elliptic curves:

\begin{enumerate}

\item[] $E_{\Lambda(\lm)}: Y^2=X(X-1)(X+(\lambda-\sqrt{{\lm}^2-1})^2)$,

\item[] $E_{{\Lambda}'(\lm)}: Y^2=X(X-1)(X+(\lambda+\sqrt{{\lm}^2-1})^2)$,
\end{enumerate}
where we denote $\Lm(\lm):=-(\lambda-\sqrt{{\lm}^2-1})^2$ and ${\Lm}'(\lm):=-(\lambda+\sqrt{{\lm}^2-1})^2$. 

  Then we have morphisms of degree two from $C_{\lm}$ to the elliptic curves.

   \begin{prop}[Ibukiyama, Katsura and Oort {\cite[\S 1.2]{IKO02}}]
    \label{propmapctoe}
    The morphism $ f_1: C_{\lm} \longrightarrow E_{\Lm(\lm)}$ sending $(x,y)$ to 
    \[
    \left(- \left(\lambda x +\frac{({\lambda}^2-1)x}{x-\lambda} \right), \frac{\epsilon {\lambda}^{\frac{3}{2}}(x-(\lambda-\sqrt{{\lm}^2-1}))}{(x-\lambda)^2} y\right),
    \]
    and the morphism $f_2: C_{\lm} \longrightarrow E_{{\Lm}'(\lm)}$ sending $(x,y)$ to 
    \[
    \left( - \left(\lambda x +\frac{({\lambda}^2-1)x}{x-\lambda} \right), \frac{\epsilon {\lambda}^{\frac{3}{2}}(x-(\lambda+\sqrt{{\lm}^2-1}))}{(x-\lambda)^2} y\right)
    \]
  are well-defined and are of degree two,
    where $\epsilon$ is a square root of $-1$
    and  ${\lm}^{3/2}$ is a square root of $\lm^3$.
    
\end{prop}

Assume that the characteristic of $K$ is greater than $2$.
The next proposition enables us to reduce the problem \textcolor{black}{of} the superspeciality of $C_\lm$ to that \textcolor{black}{of} $E_{\Lambda(\lm)}$ and $ E_{{\Lambda}'(\lm)}$.
\textcolor{black}{Again we recall that a curve $C$ over a field $K$ is \textit{superspecial} if its Jacobian variety is isomorphic to a product of supersingular elliptic curves over $\overline{K}$.}

\begin{prop}[Ibukiyama, Katsura and Oort {\cite[Lemma 1.1 (\romannumeral2),\ Proposition 1.3 (\romannumeral2), (\romannumeral3)]{IKO02}}]
\label{arigataya}
The curve $C_{\lm}$ is superspecial if and only if the two associated elliptic \textcolor{black}{curves} $E_{\Lambda(\lm)}$ and $ E_{{\Lambda}'(\lm)}$ are supersingular. 
\end{prop}

\subsection{An analytic result}
We review a result on a sum of the Legendre symbols weighted by the \textcolor{black}{von Mangoldt} function $\Lambda(n)$.

\begin{lem}
\label{jutila}
Let 
\[
S(D,X):=\sideset{}{^*} \sum_{|d| \le D} \left| \sum_{3 \le n \le X} \Lambda(n) \left( \frac{d}{n}\right) \right|,
\]
where the star on the summation indicates that $d$ is not square. Then for every $C>0$ and $3 \le D \le X^{\frac{49}{50}}$, we have 
\[
S(D,X) \ll XD(\log X)^{-C}.
\]

In addition, we have

\[
\sideset{}{^*} \sum_{0\le d \le D} \left| \  \sideset{}{'} \sum_{3 \le n \le X} \Lambda(n) \left( \frac{d}{n}\right) \right| \ll XD(\log X)^{-C},
\]
where the star on the summation indicates that $d$ is a non\textcolor{black}{-}square integer and the prime indicates that $n \equiv 3 \bmod 4$.

Moreover, for $3 \le 2D \le X^{\frac{49}{50}}$, we have
\[
\sideset{}{^{**}}\sum_{0\le d \le D} \left| \ \sideset{}{''} \sum_{3 \le n \le X} \Lambda(n) \left( \frac{d}{n}\right) \right| \ll  2XD(\log X)^{-C}.
\]
where the double star on the summation indicates that we sum over integers $d$ such that $d$ and $d/2$ are not square and the double prime indicates that $n \equiv 3 \bmod 8$. We have the same result for $n \equiv 7 \bmod 8$.
\end{lem}

\begin{proof}
    For the first part of the lemma, see Jutila \cite[Lemma 8]{Ju}.
    For the second part of the lemma, see Fouvy and Murty \cite[Lemma 6]{FM01}. We prove the last part of the lemma. We see 
\begin{eqnarray*}
    \left| \  \sideset{}{''} \sum_{3 \le n \le X} \Lambda(n) \left( \frac{d}{n}\right) \right|&=&\left|\, \sideset{}{'} \sum_{3 \le n \le X} \Lambda(n) \left( \frac{d}{n}\right) \cdot \frac{1}{2} \left(1 - \left( \frac{2}{n} \right) \right) \right|\\
    &\le& \frac{1}{2}\left|\, \sideset{}{'} \sum_{3 \le n \le X} \Lambda(n) \left( \frac{d}{n}\right) \right| +\frac{1}{2}\left|\, \sideset{}{'} \sum_{3 \le n \le X} \Lambda(n) \left( \frac{2d}{n}\right) \right|.
\end{eqnarray*}
We note that the prime on the summation on the right-hand side indicates that the sum is taken over $n \equiv 3\bmod4 $. 
Since 
\[
\sideset{}{^{**}} \sum_{0\le d \le D} \left|\, \sideset{}{'} \sum_{3 \le n \le X} \Lambda(n) \left( \frac{2d}{n}\right) \right| \le \sideset{}{^{*}} \sum_{0\le d \le 2D} \left|\, \sideset{}{'} \sum_{3 \le n \le X} \Lambda(n) \left( \frac{d}{n}\right) \right|,
\]
    we have
    \begin{eqnarray*}
       \sideset{}{^{**}} \sum_{0\le d \le D} \left| \  \sideset{}{''} \sum_{3 \le n \le X} \Lambda(n) \left( \frac{d}{n}\right) \right| \le \sideset{}{^{*}} \sum_{0\le d \le 2D} \left|\, \sideset{}{'} \sum_{3 \le n \le X} \Lambda(n) \left( \frac{d}{n}\right) \right|\ll  2XD(\log X)^{-C}.
    \end{eqnarray*}

\end{proof}


\section{The number of supersingular elliptic curves and the class numbers}\label{sec:ssing_and_classnum}
Let $p$ be a prime. In this section, we express the number of supersingular elliptic curves in a specific form with the class numbers $h(-p),h(-4p)$.
We consider the following sets related to supersingular elliptic curves, which we use in the proof of Theorem A.
We \textcolor{black}{will} use the notation \textcolor{black}{from} Sections \ref{subsec:2.1} and \ref{subsec:2.2}, for example $\Lm(\lm):=-(\lm - \sqrt{{\lm}^2-1})^2$ and ${\Lm}'(\lm):= -(\lm + \sqrt{{\lm}^2-1})^2$. In addition, $E_s$ denotes the Legendre elliptic curve given by the equation $y^2=x(x-1)(x-s)$.
\begin{enumerate}
      \item[]$S_p:=\{ {\mu}^{p-1} \mid \mu \in \fpp \smallsetminus \fp , E_{{\mu}^{p-1}} \text{ is supersingular} \}$,
    \item [] $T_p:=\{ s \in \fp\smallsetminus\{0,1 \} \mid E_s\text{ is supersingular} \}$,
    \item[] $T'_p:=T_p\smallsetminus\{-1\}$,
    \item[] $U_p:=\{-t^2 \in \fp\smallsetminus\{0,1 \} \mid t\in \fp,\ E_{-t^2} \text{ is supersingular} \}$,
    \item[] $U'_p:=U_p\smallsetminus\{-1\}$,
 \item[] ${\Theta}_p:=\left\{
 \begin{array}{cc}
       -(\lm \pm \sqrt{{\lm}^2-1})^2 
 \end{array}
\middle|  
 \begin{array}{ccc}
       \lm \in \fp \smallsetminus\{0,\ \pm1\},\\ \sqrt{{\lm}^2-1} \notin \fp ,\\
      E_{\Lm(\lm)},E_{{\Lm}'(\lm)}\text{ are supersingular}
 \end{array}
\right\}$.
\end{enumerate}
For $\mu \in \fpp \smallsetminus \fp$, let $\overline{\mu}$ be the conjugate of $\mu$ over $\fp$. Then $\overline{\mu}$ is equal to ${\mu}^p$ \textcolor{black}{since $\text{Gal}(\fpp/\fp)$ is the cyclic group of order two generated by the $p^{\text{th}}$ power map.} In addition, if \textcolor{black}{$p \equiv 3\bmod 4$}, then $T_p$ (resp.\ $U_p$) contains $-1$ so $T_p$ and $T'_p$ (resp. $U_p$ and $U'_p$) are different. 

The aim of this section is to express the cardinalities of the sets above with the class numbers. The results are as follows. 
\textcolor{black}{
\begin{prop}
    \label{classnumber}
    Assume that $p>3$.
    \begin{enumerate}
        \item [\rm{(a)}] Assume that $ p \equiv 1 \bmod 4$. Then   $\#S_p=h(-4p)$. 
        \item [\rm{(b)}] Assume that $p \equiv 3  \bmod 8 $. Then $\#U_p=3h(-p).$ 
        \item[\rm{(c)}] Assume that $ p \equiv 7 \bmod 8$. Then $\#U_p=h(-p)$.
        \item[\rm{(d)}] Assume that $p \equiv 3 \bmod 4$. Then $\#{\Theta}_p=\# U'_p$.
    \end{enumerate}
\end{prop}
}

   

   


First, we introduce an equivalence relation in the sets defined above.
\begin{defn}
\label{eqrelation}
For $s,t \in \fpp$, we say $s \approx t$ if and only if $s=t$ or $s=1/t$.
\end{defn}
 Note that if $E_s$ is $\super$, then $E_{1/s}$ is also $\super$. Thus, if $s \approx t $, then $s \in S_p$ is equivalent to $t \in S_p$. This also holds for $T_p, U_p$, and ${\Theta}_p$ respectively. For ${\Theta}_p$, we note that $\Lm(\lm)=1/{\Lm}'(\lm)$. Hence $\approx$ is an \textcolor{black}{equivalence} relation on $S_p$ (resp. $T_p$,  $U_p$, ${\Theta}_p$).


It is easy to see that for $p \ge 3$ satisfying \textcolor{black}{$p \equiv 1 \bmod 4$}, the cardinality of $S_p$ is twice of that of $S_p/\approx$, and that for $p \ge 3$ satisfying \textcolor{black}{ $p \equiv 3 \bmod 4$}, the cardinality of $T'_p$ (resp.\ $U'_p$, $\Theta_p$) is twice of that of $T'_p/\approx$ (resp.\ $U'_p/\approx$, $\Theta_p/\approx$).


\subsection{Proof of Proposition {\ref{classnumber}} (a)}
\label{kakkoitinituite}
From Lemma \ref{spsub}, we know that 
 \[
 \#\{\mu \in \fpp \smallsetminus \fp \mid E:y^2=x(x-\mu)(x-\overline{\mu}) \text{ is supersingular} \}=(p-1)h(-4p).
 \]
 Consider the map
 \[
  f: \{\mu \in \fpp \smallsetminus \fp \mid E:y^2=x(x-\mu)(x-\overline{\mu}) \text{ is supersingular} \} \longrightarrow S_p
 \]
 sending $\mu$ to $\overline{\mu}/\mu$.
 For an arbitrary $\overline{\mu}/\mu \in S_p$, we see that ${f}^{-1}(\overline{\mu}/\mu)=\{u\mu \mid u \in {\fp}^{*} \}$ and that $\#{f}^{-1}(\overline{\mu}/\mu)=p-1$.
Hence we find that $\# S_p=(p-1)h(-4p)/(p-1)=h(-4p)$.

\subsection{Proof of Proposition {\ref{classnumber}} (b)}
\begin{lem}
\label{tpup}
    Let $p$ be a prime satisfying $p \equiv3 \bmod 8$. Then $T_p=U_p$ holds.
\end{lem}
\begin{proof}
   Let $s \in T_p$. From \cite[Proposition 3.1]{AT02}, we can write $s$ as $s=-{\zeta}^8$ for some $\zeta \in \fpp$. If $\zeta \in \fp$, then we see that $s$ is in $U_p$. If $\zeta \notin \fp$, Let $\Bar{\zeta}$ be the conjugate of $\zeta$ over $\fp$. We see that
    $
    s^2=( \zeta\cdot\Bar{\zeta} )^8 $ so $ s=\pm(\zeta\cdot\Bar{\zeta})^4$.
 If $s=-(\zeta\cdot\Bar{\zeta})^4=-\left\{ (\zeta\cdot\Bar{\zeta})^2\right\}^2$, we have $s \in U_p$. Assume that $s=(\zeta\cdot\Bar{\zeta})^4$. Since $s=-{\zeta}^8$, we obtain $-1=(\Bar{\zeta}/\zeta)^4=\{{\zeta}^{p-1}\}^4$. Now we write $p=8m+3$ with $m \in \mathbb{N}$, and we have
    \[
   1= \left({\zeta}^{p-1} \right)^{p+1}= \left({\zeta}^{4(p-1)} \right)^{ 2m+1}=-1.
    \]
    This contradiction proves $T_p \subset U_p$. 
\end{proof}

\begin{proof}[Proof of Proposition \ref{classnumber} \rm{(b)}]
This follows from Lemma \ref{tpup} and Proposition \ref{numtp}.
    
\end{proof}

\subsection{Proof of Proposition{ \ref{classnumber}} (c)}

We need some preliminary steps for the proof of Proposition \ref{classnumber} (c).

\begin{lem}
\label{exist-t^2}
    Assume that $p \equiv 3\bmod 4 $, and let $s \in T_p$. Then, there exists $t\in \fp$ such that the elliptic curve $E_s$ is isomorphic to the elliptic curve $E_{-t^2}$ over $\fp$.
\end{lem}
\begin{proof}
   We know that $s$ is in the form of $-t^2$ or $t^2 $ for some $ t \in \fp$. If $s=-t^2$, the proposition is correct. Then suppose that $s=t^2$. Since $p \equiv 3 \bmod 4$ and $E_s$ is supersingular, it is easy to see that $j(E_{s}) \neq 0$. From Proposition \ref{legendrecase}, if $s-1$ is square in ${\fp}^{*}$, we find that $E_s$ is isomorphic to $E_{1/(1-s)}$. If $s-1$ is not square in ${\fp}^{*}$, we find that $E_s$ is isomorphic to $E_{s/(s-1)}$. This completes the proof.  
\end{proof}

We next show the following property of supersingular elliptic curves under the assumption that $p \equiv 7 \bmod 8$.

\begin{prop}
\label{t^2+1inf2}
    Assume that $p \equiv 7 \bmod 8$. Let $E$ be an elliptic curve $y^2=x(x-1)(x+t^2)$ with $ t \in \fp$. If E is supersingular, then $t^2+1$ is square in ${\fp}^{*}$.
\end{prop}

We use Proposition \ref{abginf2?} to prove Proposition \ref{t^2+1inf2}, so we first observe the 2-torsion group of the elliptic curve $E_{-t^2}$.

\begin{lem}
    \label{ifker2in2}
    Assume that $p \equiv 3 \bmod 4$ and let $E$ be an elliptic curve $y^2=x(x-1)(x+t^2)$ with $t \in \fp$. Let $Q_0:=(0,0),\ Q_1:=(1,0),\ Q_2:=(-t^2,0)$ be the points of $E$ of order two. Then, $Q_0 $ and $Q_2$ are not in $[2]E(\fp)$.
\end{lem}

\begin{proof}

    We first show that $Q_0$ is not in $[2]E(\fp)$. For an $\fp$-rational point $P=(x,y)$ on $E$, applying the duplication formula  $x([2]P)=(x^4-b_4x^2-2b_6x-b_8)/(4x^3+b_2x^2+2b_4x+b_6)$ (see Silverman \cite[Group Law Algorithm 2.3 (d)]{AEC}) to $b_2=-4t^2,\ b_4=-2t^2,\ b_6=0,\ b_8=-t^4$ yields
        \[
        x([2]P)= \frac{x^4+2t^2x^2+t^4}{4x^3-4t^2x^2-4t^2x}=0 .
        \]
    Then we have $x^2=-t^2$ and $x$ cannot be in $\fp$. Hence there exists no $\fp$-rational point $P$ on $E$ satisfying $[2]P=Q_0$.

    We next show that $Q_2$ is not in $[2]E(\fp)$. By the definition of group law on $E$, $Q_2=[2]P$ for some $\fp$-rational point $P$ if and only if the tangent line at $P$ on $E$ passes though $Q_2$. Note that the tangent line at $P$ is defined over $\fp$.  Let $l:y=a(x+t^2)$ with $a \in \fp$ be a line passing though $Q_2$. Substituting $y^2=a^2(x+t^2)^2$ into $y^2=x(x-1)(x+t^2)$ yields $(x+t^2)(x^2-(a^2+1)x-a^2t^2)=0$. Then $x^2-(a^2+1)x-a^2t^2$ should have double roots so that the line $l$ can be a tangent line at some point $P \in E(\fp)$. In other words, $a$ satisfies 
        $(a^2+1)^2+4a^2t^2=0$. However, since $-1$ is not square in ${\fp}^{*}$, there is no such $a\in \fp$ .    
\end{proof}

\begin{prop}
\label{p7in2}
    Assume that $p \equiv 7 \bmod 8$. Let $E$ be an elliptic curve $y^2=x(x-1)(x+t^2)$ with $t \in \fp$. Assume that $E$ is supersingular. Let $Q_0:=(0,0),\ Q_1:=(1,0),\ Q_2:=(-t^2,0)$ be the points of $E$ of order two. Then $Q_1$ is in $[2]E(\fp)$.
\end{prop}

\begin{proof}
    We write $p=8m+7$ with an integer $m$. We treat $E(\fp)$ as an abstract group and consider the quotient group $H$ of $E(\fp)$ by the subgroup $G=\{ O,Q_1\}$. Since $E$ is supersingular, we have $|E(\fp)|=p+1$. Therefore,
    $|H|=(p+1)/2=4(m+1)$.
Clearly, there exists a subgroup $\widetilde{H}$ of $H$ such that $|\widetilde{H}|=4$. For $P\in E(\fp)$, we denote the image of $P$ in $H$ by $\overline{P}$.
   
    If $\widetilde{H} \cong \mathbb{Z}/2\mathbb{Z} \times \mathbb{Z}/2\mathbb{Z}$, then there are at least  three elements of order 2 in $H$. 
    Take  $P\in E(\fp)$ such that $\overline{P}\in \widetilde{H}$ is of order two and is neither $\overline{Q_0}$ nor $\overline{Q_2}$. Note that $[2]P$ is either of $O$ or $Q_1$. If $[2]P = Q_1$, then we have the desired assertion. If $[2]P=O$, then we have $P=Q_1$ and therefore $\overline{P}$ is of order one. This is a contradiction.

    If $\widetilde{H} \cong \mathbb{Z}/4\mathbb{Z} $, we can write $\widetilde{H}=\{ O,\overline{P},[2]\overline{P},[3]\overline{P} \}$ for some $P \in E(\fp)$. Then we have $[2]([2]\overline{P})=\overline{O}$. If $[2]([2]P)=Q_1$, We get the claim. If $[2]([2]P)=O$, then $[2]P$ is in $E[2]=\{ O,Q_0,Q_1,Q_2\}$. By Lemma \ref{ifker2in2}, we get $[2]P=Q_1$.
\end{proof}
Now Proposition \ref{t^2+1inf2} immediately follows from Proposition \ref{abginf2?}.
\begin{proof}[Proof of Proposition \ref{t^2+1inf2}]
    By Proposition \ref{p7in2}, applying Proposition \ref{abginf2?} to $\alpha=0,\ \beta=-t^2,\ \gamma=1$, we see that $\gamma-\beta=1+t^2$ is square in $\fp$.
\end{proof}

Now we are ready to prove Proposition \ref{classnumber} (c).
\begin{proof}[Proof of Proposition \ref{classnumber} \rm{(c)}]
Proposition \ref{legendrecase} says that for a supersingular elliptic curve $E_s$ with $s \in T_p$, we have three elliptic curves in the Legendre form which are isomorphic to $E_s$ over $\fp$. Now, we show that exactly one of the three is in the form of $E_{-t^2}$, which indicates the cardinality of $T_p$ is equal to three times the cardinality of $U_p$.

 Let $E_{-s^2}$ and $ E_{-t^2}$ be elliptic curves with $-s^2$ and $-t^2 \in U_p$. Assume that they are isomorphic over $\fp$. Our aim is to show $-t^2=-s^2$. 
It is easy to see that $j(E_{-s^2})$ and $j(E_{-t^2})$ are non-zero.
    
    If $j(E_{-s^2})\neq 1728$, we see that $-s^2$ is neither $-1$, $2$, nor $1/2$. In addition, By Proposition \ref{t^2+1inf2}, we find that $-(s^2+1)$ is not square in $\fp$. Hence, from Proposition \ref{legendrecase}, $E_{-t^2}$ is equal to one of $E_{-s^2},\ E_{1+1/{s^2}}$, or $E_{s^2/(s^2+1)}$. However, we know that $1+1/s^2$ and $s^2/(s^2+1)$ are square in $\fp$. Then we get $-t^2=-s^2$.
    
    If $j(E_{-s^2})=1728$, then we see that $-s^2,-t^2=-1,\ 2$, or $1/2$. Since $p \equiv 7 \bmod 8$ and $2$ is square in $\fp$, we get $-s^2=-t^2=-1$.
    
    From the discussion above, we find that $\# U_p=\# T_p/3=h(-p)$.
\end{proof}

\subsection{Proof of Proposition \ref{classnumber} (d)}
\label{kakkogonituite}
We prove Proposition \ref{classnumber} (d). Assume that $p \equiv 3 \bmod 4$. Since we know that the cardinality of ${U'}_p$, we make a one-to-one correspondence between ${\Theta}_p$ and $U'_p$ to determine the cardinality of ${\Theta}_p$. Let $E$ be an elliptic curve $y^2=x^3+ax^2+bx$. Consider the elliptic curve $E': Y^2=X^3-2aX^2+(a^2-4b)X$. We have the isogeny $\theta:E \longrightarrow E'$ sending $(x,y)$ to
\[
\left( \frac{y^2}{x^2}, \frac{y(b-x^2)}{x^2} \right).
\]
This induces an isomorphism between $E/{\textrm{Ker}(\theta)}$ and $E'$ with $\text{ker}(\theta)=\{ O, (0,0)\}$ (cf.\ Silverman \cite[Proposition 10.4.9]{AEC}). We make a one-to-one correspondence between ${\Theta}_p$ and $U'_p$ through this isogeny.
For a supersingular elliptic curve $E_{-t^2}$ with $-t^2 \in U_p$, the corresponding elliptic curve by $\theta$ is
\begin{eqnarray*}
    Y^2&=&X \left(X- \left( (t^2-1)+2t\sqrt{-1} \right) \right) \left(X- \left( (t^2-1)-2t\sqrt{-1} \right) \right),
\end{eqnarray*}
where we choose a square root of $-1$, fix it and write it as $\sqrt{-1}$. Let $\nu(t):=(t^2-1)+2t\sqrt{-1} $ and its conjugate $\overline{\nu(t)}:=(t^2-1)-2t\sqrt{-1}$. Note that $E_{{\nu}^{p-1}}$ is a supersingular elliptic curve. Now we define a map
   \[
   {\Psi}:U'_p/ \approx \longrightarrow {\Theta}_{p} / \approx
   \]
 by sending the class of $-t^2$ to the class of 
    $\overline{\nu(t)}/ \nu(t) \approx \overline{\nu(-t)}/ \nu(-t)$.
We have to check that $\overline{\nu(t)}/\nu(t) \in {\Theta}_p$.
A direct calculation shows that
\[
\overline{\nu(t)}/\nu(t) = \frac{t^4-6t^2+1-4t(t^2-1)\sqrt{-1}}{(t^2+1)^2}.
\]
Let $\lm :=2t/(t^2+1)$ and we have ${\lm}^2-1=-(t^2-1)^2/(t^2+1)^2$. Then we find that 
$\overline{\nu(t)}/ \nu(t) \approx-(\lm-\sqrt{{\lm}^2-1})^2$, which is independent from the choice of $\sqrt{{\lm}^2-1}$.
Therefore $\overline{\nu(t)}/\nu(t) $ belongs to ${\Theta}_p$.
\begin{prop}
    \label{thetaup}
    The map $\Psi$ is \textcolor{black}{a} bijection.
\end{prop}
\begin{proof}
 First we show that $\Psi$ is injective. Assume $\Psi(-t^2) \approx \Psi(-s^2)$ for $-t^2,\ -s^2 \in U'_p/\approx$. If we denote $\nu:=\nu(t)$ and $ {\nu}':=\nu(s)$, we have $\overline{\nu}/\nu=\overline{{\nu}'}/{\nu}' \text{ or } \nu/\overline{\nu}=\overline{{\nu}'}/{\nu}'$. If $\overline{\nu}/\nu=\overline{{\nu}'}/{\nu}'$, then $\overline{\nu}{\nu}'=\nu \overline{{\nu}'}=\overline{\overline{\nu}{\nu}'}$ and there exists a $u\in \fp$ such that $\overline{\nu}{\nu}'=u$. Then we have $s(t^2-1)=t(s^2-1)$ \textcolor{black}{so} $(t-s)(st+1)=0$. If $t=s$, we get $-t^2=-s^2$. If $t=-1/s$, we get $-t^2=-1/s^2$. Thus $\Psi$ is injective. If $\nu/\overline{\nu}=\overline{{\nu}'}/{\nu}'$, the same discussion gives $(st-1)(s+t)=0$, whence $\Psi$ is injective.
 
Next we show that $\Psi$ is surjective. Let take the class of $-(\lm-\sqrt{{\lm}^2-1})^2 \in {\Theta}_p/\approx$.
    Note that $1-{\lm}^2$ is square in $\fp$ since ${\lm}^2-1$ is not square in $\fp$. For this $\lm$, let $t:=(-\sqrt{1-{\lm}^2}+1)/\lm$. Then a calculation shows that $\nu(t)=u( 1+\lm/\sqrt{{\lm}^2-1} )$, where we put $u:=2\sqrt{1-{\lm}^2}(\sqrt{1-{\lm}^2}-1)/{\lm}^2$. Thus we find that $\overline{\nu(t)}/\nu(t) \approx -\left( \lm -\sqrt{{\lm}^2-1} \right)^2$, and that $\Psi$ is surjective.
\end{proof}

  Now we are ready to prove Proposition \ref{classnumber} (d).
\begin{proof}[Proof of Proposition \ref{classnumber} \rm{(d)}.]
    From Proposition \ref{thetaup}, we see that the cardinality of ${\Theta}_p \ /\approx$ is equal to the cardinality of $U'_p \ /\approx$. Hence, we have $\# {\Theta}_p=\# U'_p$.
\end{proof}

\begin{cor}
\label{corsp837}
If $p \equiv 3 \bmod 8$, then $\# {\Theta}_p$ is equal to $3h(-p)-1$. If $p \equiv 7  \bmod 8$, then $\# {\Theta}_p$ is equal to $h(-p)-1$.
\end{cor}
\begin{proof}
    This follows from Proposition \ref{thetaup}, Proposition \ref{classnumber} (b), and Proposition \ref{classnumber} (c).
\end{proof}

\section{Proof of Theorem A}
In this section, we prove Theorem A. Let us restate it:

\label{theorema}
\thmA*



Our aim is to determine the cardinality of
\[
{\Sigma}_p:=\left\{ \lm \in \fp \smallsetminus\{0,\ \pm1 \} \mid  C_{\lm} \text{ is superspecial} \right\}.
\]
Note that $\# {\Sigma}_p={\psi}_{p}$.
Now we introduce an equivalence relation \textcolor{black}{on} ${\Sigma}_p$.

\begin{defn}
\label{eqrelationc}
     For $\lm,\ \llm \in {\Sigma}_p$, we say $\lm \sim \llm$ if and only if $\lm= \llm$ or $\lm=-\llm$.
\end{defn}
By Proposition \ref{arigataya}, $C_{\lm}$ is superspecial if and only if $C_{-\lm}$ is superspecial. 
Obviously, the cardinality of ${\Sigma}_p$ is twice of that of ${\Sigma}_p  / \sim$.

The following map $\Phi$ enables us to reduce our problem to
the problem of enumerating supersingular
elliptic curves. Let
\begin{eqnarray*}
 \Phi: {\Sigma}_{p}/\sim \ \longrightarrow \ \left\{ s \in \fpp \smallsetminus \{ 1,0\} \mid E_s \text{ is supersingular} \right\}/\approx 
\end{eqnarray*}
be the map sending the class of $\lm$ to the class of $\Lm(\lm) \approx {\Lm}'(\lm)$, which is well-defined by Proposition \ref{propmapctoe}. Here $\Lm(\lm)$ and $ {\Lm}'(\lm)$ are the notations defined in $\S \ref{subsec:2.2}$.
\begin{prop}
\label{injective}
    The map $\Phi$ is injective.
\end{prop}

\begin{proof}
   Assume that $\Lm(\lm) \approx \Lm(\llm)$. Then
 $\Lm(\lm)=\Lm(\llm)$ or $\Lm(\lm)=1/\Lm(\llm)=\Lm(-\llm)$. Assume that $-(\lambda-\sqrt{{\lambda}^2-1})^2=-(\llm-\sqrt{{\llm}^2-1})^2$. Then we get 
 \[
 (\lambda-\sqrt{{\lambda}^2-1}+\llm-\sqrt{{\llm}^2-1}) (\lambda-\sqrt{{\lambda}^2-1}-\llm+\sqrt{{\llm}^2-1})=0.
 \]
If $\lambda-\sqrt{{\lambda}^2-1}+\llm-\sqrt{{\llm}^2-1}=0$, then we get $\lm=-\llm$.
Similarly, if $\lambda-\sqrt{{\lambda}^2-1}-\llm+\sqrt{{\llm}^2-1}=0$, then we get $\lm=\llm$. 
Hence $\lm \sim \llm$.
\end{proof}

  

    
We determine the image of $\Phi$ by dividing into cases according to $p \bmod 8$\textcolor{black}{.}

\subsection{The case of $p \equiv 1 \bmod 4$}
We determine the image of the map $\Phi$ in the case of $p \equiv 1 \bmod 4$. 
\begin{prop}
\label{propnotinfp}
    Assume that $p \equiv 1 \bmod 4$ and $E_{\Lm(\lm)}$ is supersingular. Then ${\lambda}^2-1$ is not square in ${\fp}^{*}$.
\end{prop}
\begin{proof}
    $E_{\Lm(\lm)}$ is isomorphic to the elliptic curve $E'_{\Lm(\lm)}$ given by the equation 
    \[
    w^2=v\left(v-\left(1+\frac{\lambda}{\sqrt{{\lambda}^2-1}}\right)\right) \left(v-\left(1-\frac{\lambda}{\sqrt{{\lambda}^2-1}}\right) \right),
    \] which is defined over $\fp$. Then $E'_{\Lm(\lm)}$ also is $\super$ and we find that $\# E'_{\Lm(\lm)}(\fp)=p+1$. If $\sqrt{{\lambda}^2-1}$ is in $\fp$, then all 2-torsion points on $E'_{\Lm(\lm)}$ are $\fp$-rational and the cardinality of the two-torsion group is $4$. The two-torsion group $E'_{\Lm(\lm)}[2]$ is a subgroup of $E'_{\Lm}(\fp)$, but since $4 \nmid (p+1)$, the group $E'_{\Lm}(\fp)$ does not have a subgroup of order 4. Hence $\sqrt{{\lambda}^2-1}$ is not in $\fp$.
\end{proof}

If we denote $\mu(\lm):={\lambda}/{\sqrt{{\lambda}^2-1}}+1$ and $ {\mu}'(\lm):=\overline{\mu(\lm)}={-\lambda}/{\sqrt{{\lambda}^2-1}}+1$, we see that $\mu(\lm)$ is not in $\fp$ from Proposition \ref{propnotinfp}. In addition, we have
\begin{eqnarray*}
  &&  \Lm(\lm)=-(\lambda-\sqrt{{\lambda}^2-1})^2=\overline{\mu(\lm)} / \mu(\lm)={\mu(\lm)}^{p-1},\\
  &&  {\Lm}'(\lm)=-(\lambda+\sqrt{{\lambda}^2-1})^2={{\mu}'(\lm)}^{p-1}.
\end{eqnarray*}

Then we have the following result.
\begin{prop}
\label{resp1}
Assume that $p \equiv 1 \bmod4$. Then the image of $\Phi$ is $S_{p}/\approx$. 
\end{prop}

We use the next lemma to prove Proposition \ref{resp1}.
\begin{lem}
\label{munozyoukenn}
    Assume that $p \equiv 1 \bmod 4$. Then for every $m:=\overline{\mu}/{\mu}$ in $S_p$,
    $-(\mu-\overline{\mu})^2/\mu \overline{\mu}$ is square in ${\fp}^{*}$.
\end{lem}

\begin{proof}
    Since $E_m$ is $\super$, the elliptic curve
    $\widetilde{E_m}$ given by the equation $Y^2=X(X-\mu)(X-\overline{\mu})$ is also $\super$. We now consider the supersingular elliptic curve $\theta(\widetilde{E_m})$, which we defined in \S \ref{kakkogonituite}. The defining equation of $\theta(\widetilde{E_m})$ is
    \begin{eqnarray*}
Y^2&=&X\left(X^2+2(\mu+\overline{\mu}\right)X+\left(\mu-\overline{\mu})^2 \right)\\
&=&X \left(X-\left(-(\mu+\overline{\mu})+2\sqrt{\mu \overline{\mu}}\right) \right) \left(X- \left(-(\mu+\overline{\mu})-2\sqrt{\mu \overline{\mu}} \right) \right).
    \end{eqnarray*}
    Because of the same discussion with Proposition \ref{propnotinfp}, we find that $-(\mu+\overline{\mu}) \pm 2\sqrt{\mu \overline{\mu}}$ are not in $\fp$ and in particular, $\pm \sqrt{\mu \overline{\mu}}$ is not in $\fp$. Set $\varepsilon:=\sqrt{\mu \overline{\mu}}$. Then $\varepsilon \in \fpp \smallsetminus \fp$. Since $\fpp \supset \fp$ is an extension of degree 2, we can write $\mu=a+b\varepsilon$ with $a,b \in \fp$. Then we have
    \[
    \frac{-(\mu-\overline{\mu})^2}{\mu \overline{\mu}}=\frac{-4b^2 {\varepsilon}^2}{{\varepsilon}^2}=-4b^2.
    \]
    Since $-1$ is in ${{\fp}^{*}}^2$, we find that $-(\mu-\overline{\mu})^2/\mu \overline{\mu}$ is square in ${\fp}^{*}$.
\end{proof}

Now we are ready to prove Proposition \ref{resp1}.

\begin{proof}[Proof of Proposition \ref{resp1}]
Let $m:= \overline{\mu}/\mu \in S_p/\approx$. 
Let $\lm$ be a square root of $-(\mu-\overline{\mu})^2/4\mu \overline{\mu}$, which belongs to $\fp$. Then we have
\[
{\lm}^2-1=\frac{-(\mu+\overline{\mu})^2}{4\mu \overline{\mu}},
\]
and we get 
\[
-\left(\lm-\sqrt{{\lm}^2-1} \right)^2 \approx - \left( \frac{\overline{\mu}\sqrt{-1}}{\sqrt{\mu \overline{\mu}}} \right)^2 \approx \overline{\mu}/\mu,
\]
which is independent of the choice of $\lm$ and $\sqrt{{\lm}^2-1}$.
From Proposition \ref{arigataya}, a curve $C_{\lm}$ is superspecial for this $\lm$. Thus we have ${\Phi}(\lm) \approx m$.
\end{proof}

\begin{cor}
\label{number1}
Assume that $p \equiv 1 \bmod4$. Then, the cardinality of ${\Sigma}_{p}/\sim$ is equal to the cardinality of $S_{p}/\sim$. In addition, $\#{\Sigma}_{p}=\# S_{p}=h(-4p)$.
\end{cor}

\begin{proof}
    The first part follows from Proposition \ref{injective} and Proposition \ref{resp1}. The second part follows from Proposition \ref{classnumber} (a).
    \end{proof}

\subsection{The case of {$p \equiv 3  \bmod 8$}}
We determine the image of the map $\Phi$ in the case of $p \equiv 3 \bmod 8$. 
\begin{prop}
\label{sur3}
Assume that $p \equiv 3 \bmod 8$. Then the image of $\Phi$ is ${\Theta}_{p}/\approx \ \sqcup \  T'_p/\approx$.

\end{prop}

\begin{proof}
First, for all $\lm \in {\Sigma}_{p}$, we have $\Lm(\lm) \neq -1$ since $\lm \neq 0,1,-1$. Next, take $\Lm(\lm) \in {\Theta}_p/\approx$. Since two elliptic curves $E_{\Lm(\lm)}$ and $E_{{\Lm}'(\lm)}$ are supersingular, by Proposition \ref{arigataya}, we see that a curve
        \[
        C_{\lm}:\ y^2=x(x-1)(x+1)(x-\lm)\left(x-\frac{1}{\lm} \right)
        \]
        is superspecial. Thus for this $\lm \in {\Sigma}_p /\sim$, we have ${\Phi}(\lm) \approx \Lm(\lm)$.
        Next, take $s\in T'_p / \approx$. 
        From Lemma \ref{tpup}, there is a $t\in \fp$ such that $s=-t^2$. Set $\lm:=(t^2+1)/2t$. Then we have 
        \[
        {\lm}^2-1=\frac{(t^2-1)^2}{4t^2},
        \]
        and
        \[
        -(\lm-\sqrt{{\lm}^2-1})^2\approx -t^2 \approx s,
        \]
        which is independent of the choice of $\sqrt{{\lm}^2-1}$. By Proposition \ref{arigataya}, $C_{\lm}$ is superspecial for this $\lm$. Thus we obtain ${\Phi}(\lm) \approx s$.
\end{proof}

\begin{cor}
\label{number3}
    Assume that $p \equiv 3 \bmod 8$. Then the cardinality of ${\Sigma}_p /\sim$ is equal to the cardinality of ${\Theta}_{p} /\approx \ \sqcup \ T'_p / \approx$. In addition, 
    $\# {\Sigma}_p=6h(-p)-2$.
\end{cor}
\begin{proof}
    The first part follows by Proposition \ref{injective} and Proposition \ref{sur3}. The second part follows by Proposition \ref{numtp} and Corollary \ref{corsp837}.
\end{proof}

\subsection{The case of $p \equiv 7 \bmod 8$}
We determine the image of the map $\Phi$ in the case of $p \equiv 7 \bmod 8$. By the same discussion with the case $p \equiv 3 \bmod 8$, we obtain the following result.
\begin{prop}
\label{sur7}
Assume that $p \equiv 7 \bmod 8 $. Then the image of $\Phi$ is ${\Theta}_{p}/\approx \ \sqcup \ U'_p/\approx$. 
\end{prop}

\begin{cor}
\label{number7}
    Assume that $p \equiv 7 \bmod 8$. Then the cardinality of 
    ${\Sigma}_p /\sim$ is equal to the cardinality of ${\Theta}_{p} /\approx \ \sqcup \ U'_p / \approx$. In addition, $\# {\Sigma}_p=2h(-p)-2$.
\end{cor}
\begin{proof}
    The first part follows by Proposition \ref{injective} and Proposition \ref{sur7}. The second part follows by Proposition \ref{classnumber} (c) and Corollary \ref{corsp837}.
\end{proof}

By summarizing the results for all the above cases, we obtain Theorem A.

\section{Proof of Theorem B}
\label{theoremb}
    

We prove Theorem B in \textcolor{black}{the introduction}. We use the idea by Fouvry and Murty \cite[Theorem 4]{FM01}. 

Let $X$, $\epsilon$ and $N$ be as in Theorem B. By changing the order of summations in $\sum_{|\lm| \le N} {\phi}_{\lm}(X)$, we can express it with ${\psi}_{p}$ as follows. 
\begin{eqnarray*}
    \sum_{|\lm| \le N} {\phi}_{\lm}(X) &=&  \sum_{|\lm| \le N} \sum_{\substack{p < X \\  C_{\lm} \text{ is superspecial} }} 1 
= \sum_{p < X} \sum_{\substack{|\lm| \le N \\ C_{\lm} \text{ is superspecial} }} 1 \\
&=& \sum_{p < X}   \left( \frac{2N}{p} +O(1)  \right) {\psi}_{p} . 
\end{eqnarray*}
Since Theorem A enables us to express ${\psi}_{p}$ with the class numbers, the sum is 
\begin{align*}
\label{anaeq1}
   && \sum_{\substack{p<X \\ p \equiv 1 \bmod 4}} \frac{2N}{p} h(-4p) + 6 \sum_{\substack{p<X \\ p  \equiv 3 \bmod 8}} \frac{2N}{p} h(-p) + 2\sum_{\substack{p<X \\ p \equiv 7 \bmod 8}} \frac{2N}{p} \ h(-p) \\
   && +\  O \left( N \log \log X + X^{\frac{3}{2}} \right) .
    \stepcounter{equation}\tag{\theequation}
\end{align*}
In addition, we note the fact that 
\[
 \sum_{\substack{p<X \\ p \equiv3  \bmod 4}}h(-p)=O(X^{3/2}),
\] and
\[
 \sum_{\substack{p<X \\ p \equiv 1 \bmod 4}}h(-4p)=O(X^{3/2}).
 \]
(cf. Mertens \cite{FM} and  Siegel \cite{SI}). \textcolor{black}{By $\log (X) = o(X^\epsilon)$ and the assumption that $N$ is greater than $X^{1+\epsilon}$, the error terms of (4)} are ultimately negligible compared to $N \sqrt{X}/ \log X$. 
Recall the Dirichlet class number formula
\[
h(d)=\frac{\sqrt{|d|}}{ \pi}L(1,{\chi}_{d}),
\]
where $\chi_d$ is the character defined by the Kronecker symbol $(-d\ |\ \cdot)$.
The sum \eqref{anaeq1} is written as 
\begin{align*}
\label{anaeq2}
  && \frac{2N}{\pi} \left(  \sum_{\substack{p<X \\ p \equiv 1 \bmod 4}}\!\! \!\! \frac{2}{\sqrt{p}}L(1,{\chi}_{-4p})+ \!\! \!\! \!\! \sum_{\substack{p<X \\ p \equiv 3\bmod 8 }} \!\! \!\!  \frac{6}{\sqrt{p}}L(1,{\chi}_{-p}) + \!\! \!\! \!\!  \sum_{\substack{p<X \\ p \equiv 7 \bmod 8}} \!\! \!\!  \frac{2}{\sqrt{p}}L(1,{\chi}_{-p}) \right)\\
  &&+ o \left( \frac{N\sqrt{X}}{\log X} \right).
  \stepcounter{equation}\tag{\theequation} 
\end{align*}
In addition, from Polya's inequality (cf. Apostol \cite[Theorem 8.21]{Apo}) and Abel's identity (cf. Apostol \cite[Theorem 4.2]{Apo}), for any $1 \le m <M$,
if $p \equiv 1 \bmod 4$ we have
    \[
    \left| \sum_{m< n \le M} \frac{{\chi}_{-4p} (n)}{n} \right| < \frac{3\sqrt{4p}\log 4p}{m},
    \]
  and if $p \equiv 3 \bmod 4$, we have 
     \[
    \left| \sum_{m< n \le M} \frac{{\chi}_{-p} (n)}{n} \right| < \frac{3\sqrt{p}\log p}{m}.
    \]
Hence for any parameter $U>1$, we have
\begin{eqnarray}
\label{polyabel}
     \left| r_2(U,-p) \right| < \frac{3\sqrt{p}\log p}{U}, \quad \left| r_2(U,-4p)\right| < \frac{3\sqrt{4p}\log 4p}{U},
\end{eqnarray}
where 
\begin{eqnarray*}
    r_2(U,-p)=L(1,{\chi}_{-p})- \sum_{n \le U}\frac{{\chi}_{-p}(n)}{n},\\
    r_2(U,-4p)=L(1,{\chi}_{-4p})- \sum_{n \le U}\frac{{\chi}_{-4p}(n)}{n}.
\end{eqnarray*}
We choose $U=X^{3/4}$. Then \eqref{anaeq2} is as follow\textcolor{black}{s}.
\begin{align*}
\label{anaeq3}
   && \frac{2N}{\pi} \cdot 2 \sum_{\substack{p<X \\ p \equiv 1 \bmod 4}}\frac{1}{\sqrt{p}} \sum_{n \le U}\frac{{\chi}_{-4p}(n)}{n}+  \frac{2N}{\pi} \cdot 6 \sum_{\substack{p<X \\ p \equiv 3 \bmod 8 }} \frac{1}{\sqrt{p}}\sum_{n \le U} \frac{{\chi}_{-p}(n)}{n} \\ 
   && + \frac{2N}{\pi} \cdot 2\sum_{\substack{p<X \\ p \equiv 7 \bmod 8 }}\frac{1}{\sqrt{p}}\sum_{n \le U}\frac{{\chi}_{-p}(n)}{n} 
    + o \left( \frac{N\sqrt{X}}{\log X} \right)
     \stepcounter{equation}\tag{\theequation} .
\end{align*}

\begin{rem}
Using the inequalities \eqref{polyabel}, 
we see that the error terms from $r_2(U,-p)$ and $r_2(U,-4p)$ are smaller than 
\[
C \cdot N \sum_{p < X} \frac{\log 4p}{U}
\]
for a certain constant $C$. Since $U=X^{3/4}$, we see that the error terms are ultimately negligible compared to $N\sqrt{X}/ \log X$.
    
\end{rem}
We recall that ${\chi}_{-4p}$ and ${\chi}_{-p}$ are expressed by the Legendre symbols: 
 if $p \equiv 1\bmod 4$\textcolor{black}{, then}
    \[
    {\chi}_{-4p}= \left( \frac{2}{n} \right)^2(-1)^{\frac{p-1}{2} \cdot \frac{n-1}{2}} \left( \frac{n}{p} \right),
    \]
if $p \equiv 3 \bmod 4$\textcolor{black}{, then}
     \[
    {\chi}_{-p}= \left( \frac{n}{p} \right).
    \]
Substituting these \textcolor{black}{in} \eqref{anaeq3} yields
\begin{align*}
\label{anaeq4}
   && \frac{4N}{\pi}S_1(X,U) +  \frac{12N}{\pi} S_3(X,U)
   + \frac{4N}{\pi}  S_7(X,U)
    + o \left( \frac{N\sqrt{X}}{\log X} \right),
     \stepcounter{equation}\tag{\theequation} 
\end{align*}
where
\begin{eqnarray*}
    S_1(X,U)&:=& \sum_{n \le U} \frac{1}{n} \sum_{\substack{p<X \\ p \equiv1 \bmod 4}}\frac{\left( \frac{2}{n} \right)^2(-1)^{\frac{p-1}{2} \cdot \frac{n-1}{2}} \left( \frac{n}{p} \right)}{\sqrt{p}},\\
    S_3(X,U)&:=&\sum_{n \le U} \frac{1}{n} \sum_{\substack{p<X \\ p \equiv 3\bmod 8}}\frac{\left( \frac{n}{p} \right)}{\sqrt{p}},\\
    S_7(X,U)&:=&\sum_{n \le U} \frac{1}{n} \sum_{\substack{p<X \\ p \equiv7 \bmod 8}}\frac{\left( \frac{n}{p} \right)}{\sqrt{p}}.\\
\end{eqnarray*}
We calculate these sums using the idea of Fouvry and Murty \cite[Theorem 4]{FM01}.

We first calculate $S_3(X,U)$ (resp.\ $S_7(X,U)$). If we sum over only $n \le U$ such that $n$ is a perfect square or $n/2$ is a perfect square, then we have 
    \begin{eqnarray*}
    \sum_{\substack{n \le U \\ n=d^2 \text{ for some } d \in \bN }} \frac{1}{d^2} \sum_{\substack{p<X \\ p \equiv 3 \bmod 8 }} \frac{1}{\sqrt{p}}-\sum_{\substack{n \le U \\ n=2l^2 \text{ for some } l \in \bN }} \frac{1}{2l^2} \sum_{\substack{p<X \\ p \equiv 3 \bmod 8}} \frac{1}{\sqrt{p}}.
    \end{eqnarray*}
we note that $2 \notin {{\fp}^{*}}^2$ if $p \equiv 3 \bmod 8$. 
We know that the inner sum of each term is equal to
\[
\frac{2}{\varphi(8)} \frac{\sqrt{X}}{\log X} +o \left( \frac{\sqrt{X}}{\log X} \right),
\]
where $\varphi$ is Euler's totient function. Next, we can calculate 
\begin{eqnarray*}
    \sum_{\substack{n \le U \\ n=d^2 \text{ for some } d \in \bN }} \frac{1}{d^2}=\frac{{\pi}^2}{6}+O\left( \frac{1}{X^{3/8}} \right).
\end{eqnarray*}
Thus, we get
\[
 \sum_{\substack{n \le U \\ n=d^2 \text{ for some } d \in \bN }} \frac{1}{d^2} \sum_{\substack{p<X \\ p \equiv 3 \bmod 8 }} \frac{1}{\sqrt{p}}=\frac{{\pi}^2}{12} \frac{\sqrt{X}}{\log X} +o \left( \frac{\sqrt{X}}{\log X} \right).
\]
Similarly, we get 
\[
\sum_{\substack{n \le U \\ n=2l^2 \text{ for some } l \in \bN }} \frac{1}{2l^2} \sum_{\substack{p<X \\ p \equiv 3 \bmod 8 }} \frac{1}{\sqrt{p}}=\frac{{\pi}^2}{24} \frac{\sqrt{X}}{\log X} +o \left( \frac{\sqrt{X}}{\log X} \right).
\]
\textcolor{black}{Following a similar approach to} Fouvry and Murty \cite[p.\ 92]{FM01} and using Lemma \ref{jutila}, the sum of the remaining terms is 
\[
\sum_{\substack{n \le U \\ n \text{ and } n/2 \text{ is not square}}} \frac{1}{n} \sum_{\substack{p<X \\ p \equiv 3 \bmod 8}}\frac{\left( \frac{n}{p} \right)}{\sqrt{p}} =o \left( \frac{\sqrt{X}}{\log X} \right).
\]
Hence we obtain
\begin{eqnarray}
\label{3mod8}
    S_3(X,U)=\frac{{\pi}^2}{24} \frac{\sqrt{X}}{\log X} +o \left( \frac{\sqrt{X}}{\log X} \right).
\end{eqnarray}

In the case of $S_7(X,U)$, since $2 \in {{\fp}^{*}}^2$ if $p \equiv 7 \bmod 8$, the same discussion yields
\begin{eqnarray}
\label{7mod8}
    S_7(X,U)=\frac{{\pi}^2}{8} \frac{\sqrt{X}}{\log X} +o \left( \frac{\sqrt{X}}{\log X} \right).
\end{eqnarray}

Second, we calculate $S_1(X,U)$. If $n$ is even, then ${\chi}_{-4p}(n)$ is $0$ so we only treat odd $n$. If we sum over only odd $n \le U$ such that $n$ is perfect square, then we have
  \begin{eqnarray*}
    \sum_{\substack{n \le U \\ n=(2d-1)^2 \text{ for some } d \in \bN }} \frac{1}{(2d-1)^2} \sum_{\substack{p<X \\ p \equiv 1  \bmod 4}} \frac{1}{\sqrt{p}}.
    \end{eqnarray*}
    We know that 
    \begin{eqnarray*}
     && \sum_{\substack{p<X \\ p\equiv 1 \bmod 4}} \frac{1}{\sqrt{p}}=  \frac{2}{\varphi(4)} \frac{\sqrt{X}}{\log X} +o \left( \frac{\sqrt{X}}{\log X} \right),\\
     && \sum_{\substack{n \le U \\ n=(2d-1)^2 \text{ for some } d \in \bN }} \frac{1}{(2d-1)^2}=\frac{{\pi}^2}{8}+O\left( \frac{1}{\sqrt{U}} \right).
    \end{eqnarray*}
    Thus we obtain 
    \begin{eqnarray*}
         \sum_{\substack{n \le U \\ n=(2d-1)^2,\ \exists d \in \bN }} \frac{1}{(2d-1)^2} \sum_{\substack{p<X \\ p \equiv 1 \bmod 4 }} \frac{1}{\sqrt{p}}=\frac{{\pi}^2}{8}\frac{\sqrt{X}}{\log X} +o \left( \frac{\sqrt{X}}{\log X} \right).
    \end{eqnarray*}
\textcolor{black}{Following a similar approach to} Fouvry and Murty \cite[p.92]{FM01}, the sum of the remaining terms is 
\[
\sum_{\substack{n \le U \\ \text{odd }n\text{ is not square}}} \frac{1}{n} \sum_{\substack{p<X \\ p \equiv 1 \bmod 4}}\frac{\left( \frac{2}{n} \right)^2 (-1)^{\frac{p-1}{2} \cdot \frac{n-1}{2}} \left( \frac{n}{p} \right)}{\sqrt{p}} =o \left( \frac{\sqrt{X}}{\log X} \right).
\]
Thus we find that
\begin{eqnarray}
\label{1mod4}
    S_1(X,U)=\frac{{\pi}^2}{8}\frac{\sqrt{X}}{\log X} +o \left( \frac{\sqrt{X}}{\log X} \right).
\end{eqnarray}
Substituting the equations \eqref{3mod8}, \eqref{7mod8}, \eqref{1mod4} for \eqref{anaeq4} yields 
\[
 \sum_{|\lm| \le N} {\phi}_{\lm}(X) =\frac{3\pi N}{2}\frac{\sqrt{X}}{\log X} +o \left( \frac{N \sqrt{X}}{\log X} \right).
\]
Hence we obtain
\[
 \frac{1}{N}\sum_{|\lm| \le N} {\phi}_{\lm}(X) =\frac{3\pi }{2}\frac{\sqrt{X}}{\log X} +o \left( \frac{ \sqrt{X}}{\log X} \right).
\]

\section{Proof of Theorem C}
\label{theoremc}

We prove Theorem C in \textcolor{black}{the introduction}. To show Theorem C, for each $t=0,1, \ldots,p-1$ we calculate the number of $\lm \in \bQ $ \textcolor{black}{with} height less than $N$ and \textcolor{black}{$\lm  \equiv t \bmod p$}. 

We take a positive $\lm=a/b \in \bQ$ with positive coprime integers $a,b$. Assume that $\height (\lm) \le N$. Then we have $\textcolor{black}{1} \le a\le N$ and $1 \le b \le N$. It is easy to see that \textcolor{black}{$a/b  \equiv t \bmod p$} if and only if $b$ is not divisible by $p$ and $a-bt =kp$ for some $k \in \bZ$. We fix $b$ that is not divisible by $p$ and we calculate the number of positive irreducible fractions $\lm=a/b>0$ satisfying \textcolor{black}{$\lm \equiv t \bmod p$}.

If $b=1$, then it is equal to $N/p + O(1)$. If $b \ge 2$, then we have $k \neq 0$ since $a$ and $b$ are relatively prime. Since $a$ satisfies $\textcolor{black}{1} \le a \le N$, we have $\textcolor{black}{(1-bt)/p} \le k \le (N-bt)/p$. It is easy to see that $a$ and $b$ are relatively prime if and only if $k$ and $b$ are relatively prime. Then we count the number of such $k$, which is
\[
\sum_{d \mid b}\mu(d) \frac{N}{pd} +\textcolor{black}{O(d(b))}.
\]
Here, $\mu$ is the M\"{o}bius function \textcolor{black}{and $d(b)$ is the number of divisors of $b$.} We see that $\sum_{b\le N}d(b) \sim N\log N$.

\noindent
Then we find that $\# \{ \lm \in \bQ \mid \lm >0,\ \height(\lm) \le N,\ \lm \equiv t \bmod p\}$ is 
\[
 \frac{N}{p} + \frac{N}{p} \sum_{\substack{2 \le b \le N \\ p \nmid b}} \frac{\varphi(b)}{b} + O(N \log N).
\]
We calculate the sum of $\varphi(b)/b$. We see that 
\[
\sum_{\substack{2 \le b \le N \\ p \nmid b}} \frac{\varphi(b)}{b} =\sum_{2 \le b \le N} \frac{\varphi(b)}{b}- \sum_{\substack{2 \le b \le N \\ p \mid b}} \frac{\varphi(b)}{b}.
\]
\textcolor{black}{From Apostol \cite[\S 3 Exercise 5]{Apo}, we have} 
\begin{eqnarray*}
    \sum_{2 \le b \le N} \frac{\varphi(b)}{b} &=&  \frac{6}{{\pi}^2}N +\textcolor{black}{ O(\log N)},
\end{eqnarray*}
and therefore
\begin{eqnarray*}
    \sum_{\substack{2 \le b \le N \\ p \mid b}} \frac{\varphi(b)}{b}
    = \sum_{1 \le l \le \frac{N}{p}} \frac{\varphi(lp)}{lp}
    \le \sum_{1 \le l \le \frac{N}{p}} \frac{\varphi(l)}{l}= \frac{6}{{\pi}^2}\frac{N}{p} + \textcolor{black}{O(\log N)}.
\end{eqnarray*}
Then we find that \textcolor{black}{
\[
\# \{ \lm \in \bQ \mid \lm >0,\ \height (\lm) \le N,\ \lm \equiv t \bmod p\}= \frac{6}{{\pi}^2} \frac{N^2}{p} + O(N\log N) +O\left( \frac{N^2}{p^2} \right).
\]}
\textcolor{black}{This asymptotic approximation does not depend on $t$. Hence, including the contributions from negative $\lm$, we obtain}
\[
\# \{ \lm \in \bQ \mid  \height (\lm) \le N,\ \lm \equiv t \bmod p\}= \frac{12}{{\pi}^2} \frac{N^2}{p} + O(N\log N) +O\left( \frac{N^2}{p^2} \right).
\]
The calculation above gives
\begin{eqnarray*}
    \sum_{\text{ht}(\lm) \le N} {\phi}_{\lm}(X) &=& 
\sum_{p < X} \sum_{\substack{\text{ht}(\lm) \le N \\ C_{\lm} \text{ is superspecial} }} 1\\
&=& \sum_{p < X} \left(  \frac{12}{{\pi}^2} \frac{N^2}{p} + O(N\log N) +O\left( \frac{N^2}{p^2} \right) \right) {\psi}_p.
\end{eqnarray*}
A calculation shows that the error terms are ultimately negligible compared to $N^2 \sqrt{X} /\log X$. Thus, from the result of Theorem B, we find that
\[
\sum_{\lm \in \bQ,\ \height(\lm) \le N } {\phi}_{\lm}(X) = \frac{9}{\pi} \frac{N^2\sqrt{X}}{\log X} +o \left( \frac{N^2\sqrt{X}}{\log X}\right).
\]
Hence we obtain
 \[
    \frac{1}{N^2}\sum_{\lm \in \bQ,\ \height(\lm) \le N } {\phi}_{\lm}(X) = \frac{9}{\pi} \frac{\sqrt{X}}{\log X} +o \left( \frac{\sqrt{X}}{\log X}\right).
 \]

\subsection*{Acknowledgements}
This paper was written under the supervision of Professor Shushi Harashita.
The author thanks him for his consistent support. 
The author also thank Toshiyuki Katsura for his helpful comments.
The author is grateful to the anonymous referees for their careful reading of this article and their comments. 
This research was supported by JSPS Grant-in-Aid for Scientific Research (C) 21K03159.
 
\subsection*{Data availability statement} 
Data sharing is not applicable to this article, as no datasets were generated or analyzed during the present work.

\subsection*{Conflict of interest}
The authors declare no conflicts of interest associated with this manuscript.

\end{document}